%% file: smallu.tex
\documentclass[a4paper]{amsart}
\author{Dilip Raghavan}
\thanks{First author was partially supported by the Singapore Ministry of Education's research grant number MOE2017-T2-2-125.}
\address{Department of Mathematics \\
National University of Singapore\\
Singapore 119076}
\email{dilip.raghavan@protonmail.com}
\urladdr{http://www.math.nus.edu.sg/$\sim$raghavan}
\author{Saharon Shelah}
\thanks{Both authors were partially supported by European Research Council grant 338821.
Publication 1160 on Shelah's list.}
\date{\today}
\title[Small ultrafilter number]{A small ultrafilter number at smaller cardinals}
\usepackage[only, llbracket, rrbracket]{stmaryrd}
\usepackage{amssymb, amsmath, amsthm, mathrsfs, enumitem, amsfonts, latexsym, bbm}
\def\polhk#1{\setbox0=\hbox{#1}{\ooalign{\hidewidth
    \lower1.5ex\hbox{`}\hidewidth\crcr\unhbox0}}}
\newtheorem{Theorem}{Theorem}
\newtheorem{Claim}[Theorem]{Claim}
\newtheorem{Lemma}[Theorem]{Lemma}
\newtheorem{Cor}[Theorem]{Corollary}

\newtheorem{Question}[Theorem]{Question}

\theoremstyle{definition}
\newtheorem{Def}[Theorem]{Definition}

\theoremstyle{remark}

\newcommand{\forces}{\Vdash}

\newcommand{\dd}{{\mathfrak{d}}}
\newcommand{\rr}{{\mathfrak{r}}}

\newcommand{\inva}{{\mathfrak{a}}}
\newcommand{\uu}{{\mathfrak{u}}}

\renewcommand{\[}{\left[}
\renewcommand{\]}{\right]}
\newcommand{\PP}{\mathbb{P}}

\newcommand{\QQ}{\mathbb{Q}}
\newcommand{\lc}{\left|}
\newcommand{\rc}{\right|}

\newcommand\ZFC{\mathrm{ZFC}}

\newcommand\GCH{\mathrm{GCH}}
\newcommand\Fn{\mathrm{Fn}}

\DeclareMathOperator{\cf}{cf}

\newcommand{\DDD}{\mathcal{D}}

\newcommand{\UUU}{{\mathcal{U}}}

\newcommand{\FFF}{{\mathcal{F}}}

\newcommand{\V}{{\mathbf{V}}}

\newcommand{\VG}{{{\mathbf{V}}[G]}}
\newcommand{\VP}{{\mathbf{V}}^{\PP}}

\newcommand{\cs}{{\subseteq}_{\mathrm{c}}}
\newcommand{\RR}{\mathbb{R}}

\newcommand{\XXX}{\mathcal{X}}

\newcommand{\pr}[2]{\langle #1, #2 \rangle}
\newcommand{\seq}[4]{\langle {#1}_{#2}: #2 #3 #4 \rangle}

\begin{document}
\begin{abstract}
 It is proved to be consistent relative to a measurable cardinal that there is a uniform ultrafilter on the real numbers which is generated by fewer than the maximum possible number of sets.
 It is also shown to be consistent relative to a supercompact cardinal that there is a uniform ultrafilter on ${\aleph}_{\omega+1}$ which is generated by fewer than ${2}^{{\aleph}_{\omega+1}}$ sets.
\end{abstract}
\maketitle
\section{Introduction} \label{sec:intro}
The purpose of this short note is to show that it is possible to make the ultrafilter number small at relatively small accessible regular cardinals assuming the existence of large cardinals.
Recall the following definitions.
\begin{Def} \label{def:uniform}
 Let $\kappa \geq \omega$ be a regular cardinal.
 An ultrafilter $\UUU$ on $\kappa$ is said to be \emph{uniform} if $\lc A \rc = \kappa$ for every $A \in \UUU$.
 A set $\XXX \subseteq \UUU$ \emph{generates} $\UUU$ if
 \begin{align*}
  \UUU = \left\{ A \subseteq \kappa: \exists B \in \XXX \[B \subseteq A\] \right\}.
 \end{align*}
 The cardinal $\uu(\kappa)$ is defined to be smallest size of a family that generates a uniform ultrafilter on $\kappa$.
 More formally,
 \begin{align*}
  \uu(\kappa) = \min\left\{ \lc \XXX \rc: \XXX \ \text{generates some uniform ultrafilter on} \ \kappa \right\}.
 \end{align*} 
\end{Def}
Much is known about $\uu(\omega)$.
The consistency of $\uu(\omega) = {\aleph}_{1} < {2}^{{\aleph}_{0}}$ seems to have been first noted by Kunen in the early 1970s (see Exercise (A10) of Chapter VIII in \cite{Kunen}).
To obtain Kunen's model, one starts with ${\aleph}_{1} < {2}^{{\aleph}_{0}}$ and then adjoins an ultrafilter witnessing $\uu(\omega) = {\aleph}_{1}$ by a finite support iteration of c.c.c.\@ forcings of length ${\omega}_{1}$.
Baumgartner and Laver (see \cite{iteratedsacks}) noticed later that both countable support iterations and countable support products of Sacks forcing preserve P-points, and hence they showed that $\uu(\omega) = {\aleph}_{1} < {\aleph}_{2} = {2}^{{\aleph}_{0}}$ holds in both the iterated and side-by-side Sacks models.
The Miller model (see \cite{millerforcing} and \cite{BJ}) provides an example of a model where $\uu(\omega) = {\aleph}_{1} < {\aleph}_{2} = \dd(\omega)$ holds.
Much later, Shelah proved the consistency of $\uu(\omega) < \inva(\omega)$ assuming the consistency of a measurable cardinal in \cite{sh700}.

The situation above $\omega$ is much less clear.
Kunen asked in the seventies whether $\uu({\aleph}_{1}) < {2}^{{\aleph}_{1}}$ or even whether $\rr({\aleph}_{1}) < {2}^{{\aleph}_{1}}$ is consistent.
Kunen's questions remain completely open.
If $\kappa \geq {\beth}_{\omega}$ is regular, then $\dd(\kappa) \leq \rr(\kappa)$ (see \cite{abrd}), and therefore $\dd(\kappa) \leq \rr(\kappa) \leq \uu(\kappa)$.
Hence there can be no perfect analogue of the Miller model at sufficiently large regular cardinals.
It remains an open problem whether $\dd(\kappa) \leq \rr(\kappa)$ is provable for uncountable regular cardinals $\kappa$ less than ${\beth}_{\omega}$. 
For regular $\kappa > \omega$, the consistency of $\uu(\kappa) < {2}^{\kappa}$ was only known for supercompact $\kappa$ until now.
An unpublished result of Carlson from the eighties showed that if $\kappa$ is a Laver indestructible supercompact cardinal, then there is a forcing extension in which $\kappa$ remains supercompact and $\uu(\kappa) = {\kappa}^{+} < {2}^{\kappa}$.
Carlson's model is obtained in a manner analogous to how Kunen's model for $\uu(\omega) = {\aleph}_{1} < {2}^{{\aleph}_{0}}$ is obtained.

In this note, we produce models where $\uu(\kappa) < {2}^{\kappa}$ for accessible values of $\kappa$.
More precisely, assuming a measurable cardinal in the ground model, we produce a model where ${2}^{{\aleph}_{0}}$ is regular and $\uu({2}^{{\aleph}_{0}}) < {2}^{{2}^{{\aleph}_{0}}}$, and assuming a supercompact cardinal in the ground model we produce a model where $\uu({\aleph}_{\omega+1}) < {2}^{{\aleph}_{\omega+1}}$.
We do not know if any large cardinals are necessary to produced models satisfying these statements.
Our models are unlikely to be optimal in several other ways.
For instance in all of our models, ${2}^{\kappa}$ is much larger than ${\kappa}^{+}$.
At present, we do not know how to produce models of $\uu(\kappa) < {2}^{\kappa}$ for accessible values of $\kappa$ where the gap between ${\kappa}^{+}$ and ${2}^{\kappa}$ is small.
See Section \ref{sec:q} for further discussion of open problems.

Most of the ideas needed to prove our theorems come from a paper of Shelah and Thomas~\cite{sh:304} in which several statements about subgroups of the symmetric group on $\kappa$ were shown to be consistent relative to large cardinals.
In fact we show that $\uu({\aleph}_{\omega+1}) < {2}^{{\aleph}_{\omega+1}}$ holds in the model constructed in Section 4 of \cite{sh:304}.
A crucial ingredient used in the proofs in our paper and in the paper of Shelah and Thomas~\cite{sh:304} is the notion of an indecomposable filter.
In particular we will use a theorem of Ben-David and Magidor~\cite{indecomposable} saying that indecomposable filters may exist on ${\aleph}_{\omega+1}$.

While we will only consider $\uu(\kappa)$ for regular $\kappa$ in this note, several other works such as Garti and Shelah~\cite{singularu} have dealt with the ultrafilter number at singular cardinals.      
\section{A general result} \label{sec:general}
In this section we will present a general theorem saying that if $\mu$ is a singular strong limit cardinal, if $\lambda$ and $\kappa$ are specifically chosen cardinals below $\mu$, and if $\PP$ is any forcing notion that satisfies a combinatorial condition relative to $\lambda, \kappa$, and $\mu$, then $\PP$ forces that $\uu(\kappa) \leq \mu$.
We will apply this general result in Sections \ref{sec:uc} and \ref{sec:alephomega+1} to obtain consistency results. 
\begin{Def} \label{def:cs}
 Let $\pr{\PP}{{\leq}_{\PP}, {\mathbbm{1}}_{\PP}}$ and $\pr{\QQ}{{\leq}_{\QQ}, {\mathbbm{1}}_{\QQ}}$ be notions of forcing.
 We will write $\pr{\PP}{{\leq}_{\PP}, {\mathbbm{1}}_{\PP}} \: \cs \: \pr{\QQ}{{\leq}_{\QQ}, {\mathbbm{1}}_{\QQ}}$ if the following conditions are satisfied:
 \begin{enumerate}
  \item
  ${\mathbbm{1}}_{\PP} = {\mathbbm{1}}_{\QQ}$;
  \item
  $\PP \subseteq \QQ$;
  \item
  ${\leq}_{\PP} = {\leq}_{\QQ} \cap \left( \PP \times \PP \right)$;
  \item
  for any $p, p' \in \PP$, $p \: {\perp}_{\PP} \: p' \iff p \: {\perp}_{\QQ} \: p'$;
  \item
  if $A \subseteq \PP$ is any maximal antichain in $\pr{\PP}{{\leq}_{\PP}, {\mathbbm{1}}_{\PP}}$, then $A$ remains a maximal antichain in $\pr{\QQ}{{\leq}_{\QQ}, {\mathbbm{1}}_{\QQ}}$. 
 \end{enumerate}
The relation $\pr{\PP}{{\leq}_{\PP}, {\mathbbm{1}}_{\PP}} \: \cs \: \pr{\QQ}{{\leq}_{\QQ}, {\mathbbm{1}}_{\QQ}}$ is usually expressed by saying that $\pr{\PP}{{\leq}_{\PP}, {\mathbbm{1}}_{\PP}}$ is a \emph{complete suborder of} $\pr{\QQ}{{\leq}_{\QQ}, {\mathbbm{1}}_{\QQ}}$.
We also usually abuse notation and simply write $\PP \: \cs \: \QQ$ or say that $\PP$ is a complete suborder of $\QQ$. 
\end{Def}
It is clear that $\cs$ is a transitive relation.
The following simple fact will be useful.
\begin{Lemma} \label{lem:product}
 Let $\pr{\PP}{{\leq}_{\PP}, {\mathbbm{1}}_{\PP}}$, $\pr{\QQ}{{\leq}_{\QQ}, {\mathbbm{1}}_{\QQ}}$, and $\pr{\RR}{{\leq}_{\RR}, {\mathbbm{1}}_{\RR}}$ be any forcing notions.
 If $\PP \: \cs \: \QQ$, then $\PP \times \RR \: \cs \: \QQ \times \RR$.
\end{Lemma}
\begin{proof}
 Points (1)--(4) of Definition \ref{def:cs} are clear.
 For (5), consider any $\pr{q}{r} \in \QQ \times \RR$.
 Since $\PP \: \cs \: \QQ$, there exists $p \in \PP$ with the property that $\forall p' \: {\leq}_{\PP} \: p\[p' \: {\not\perp}_{\QQ} \: q \]$.
 Now $\pr{p}{r} \in \PP \times \RR$.
 Moreover if $\pr{p'}{r'}$ is any condition such that $\pr{p'}{r'} \: {\leq}_{\PP \times \RR} \: \pr{p}{r}$, then $\pr{p'}{r'} \: {\not\perp}_{\QQ \times \RR} \: \pr{q}{r}$.
 This implies (5).
\end{proof}
The notion of an indecomposable filter was first considered by Prikry~\cite{prikrythesispaper} and investigated by many other afterwards.
Ben-David and Magidor~\cite{indecomposable} showed that it is consistent relative to a supercompact cardinal that uniform ${\aleph}_{n}$-indecomposable ultrafilters can exist on ${\aleph}_{\omega+1}$ for $0 < n < \omega$.
This is the key combinatorial notion needed for our proofs.
\begin{Def} \label{def:indec}
 Let $\kappa$ and $\lambda$ be infinite cardinals.
 A filter $\FFF$ on $\lambda$ is said to be \emph{$\kappa$-indecomposable} if whenever $\seq{Y}{\xi}{<}{\kappa}$ is a partition of $\lambda$ (i.e.\@ $\lambda = {\bigcup}_{\xi < \kappa}{{Y}_{\xi}}$ and $\forall \zeta < \xi < \kappa\[{Y}_{\zeta} \cap {Y}_{\xi} = 0\]$), then there exists $T \subseteq \kappa$ such that $\lc T \rc < \kappa$ and ${\bigcup}_{\xi \in T}{{Y}_{\xi}} \in \FFF$.
\end{Def}
We next introduce a technical combinatorial condition on a forcing notion $\PP$ involving several other parameters.
In Section \ref{sec:uc} it will be proved that forcing notions of the form $\Fn(I, J, \lambda)$ and products of forcing notions of this form satisfy this combinatorial condition for a suitable choice of the other parameters.
In this section, we will prove that if $\PP$ satisfies the combinatorial condition for some choice of the other parameters, then $\PP$ forces that the ultrafilter number at one of these parameters is bounded by another parameter. 
\begin{Def} \label{def:filtration}
 Let $\pr{\PP}{{\leq}_{\PP}, {\mathbbm{1}}_{\PP}}$ be a forcing notion.
 We say that $\pr{\PP}{{\leq}_{\PP}, {\mathbbm{1}}_{\PP}}$ has a \emph{$(\lambda, \kappa, \mu, \DDD)$-filtration} if there exists a sequence $\langle {\PP}_{\alpha}: \alpha < \mu \rangle$ satisfying the following:
 \begin{enumerate}
  \item
  $\lambda, \kappa$, and $\mu$ are infinite cardinals satisfying $\lambda < \cf(\mu) < \kappa < \mu$;
  \item
  $\mu$ is a strong limit cardinal and ${\lambda}^{< \lambda} = \lambda$;
  \item
  $\DDD$ is a uniform $\cf(\mu)$-indecomposable filter on $\kappa$;
  \item
  $\PP$ is ${\lambda}^{+}$-c.c.\@ and $\forall p \in \PP \exists \alpha < \mu\[p \in {\PP}_{\alpha}\]$;
  \item
  for each $\alpha < \mu$, ${\PP}_{\alpha}\: \cs \: \PP$, and $\forall \xi < \alpha\[{\PP}_{\xi} \subseteq {\PP}_{\alpha}\]$;
  \item
  for each $\alpha < \mu$, $\lc {\PP}_{\alpha} \rc < \mu$.
 \end{enumerate} 
\end{Def}
Observe that there is no connection between $\PP$ and the filter $\DDD$.
In other words, we only need the existence of some uniform $\cf(\mu)$-indecomposable filter on $\kappa$.
Conditions (4)--(6) simply say that $\PP$ is a ${\lambda}^{+}$-c.c.\@ poset which can be written as an increasing union of small complete subposets.
Actually the condition that ${\PP}_{\alpha}$ is a complete sub order of $\PP$ for each $\alpha < \mu$ is not necessary for the proof of our main results.
It is sufficient if each ${\PP}_{\alpha}$ is any sub order of $\PP$.
However Condition (5) is automatically satisfied in all of our applications.
Hence we have not sought to weaken it.  
\begin{Lemma} \label{lem:productfiltration}
 Let $\pr{\PP}{{\leq}_{\PP}, {\mathbbm{1}}_{\PP}}$ and $\pr{\RR}{{\leq}_{\RR}, {\mathbbm{1}}_{\RR}}$ be forcing notions.
 Assume that $\lambda, \kappa, \mu$, and $\DDD$ are so that $\pr{\PP}{{\leq}_{\PP}, {\mathbbm{1}}_{\PP}}$ has a $(\lambda, \kappa, \mu, \DDD)$-filtration.
 If $\lc \RR \rc < \mu$ and
 \begin{align*}
  {\forces}_{\PP}{\text{``} \: \check{\RR} \ \text{is} \ {\check{\lambda}}^{+}\text{-c.c.\@''}},
 \end{align*} 
 then $\PP \times \RR$ also has a $(\lambda, \kappa, \mu, \DDD)$-filtration.  
\end{Lemma}
\begin{proof}
 Fix $\seq{\PP}{\alpha}{<}{\mu}$ witnessing that  $\pr{\PP}{{\leq}_{\PP}, {\mathbbm{1}}_{\PP}}$ has a $(\lambda, \kappa, \mu, \DDD)$-filtration.
 For each $\alpha < \mu$, let ${\QQ}_{\alpha} = {\PP}_{\alpha} \times \RR$.
 We check that $\seq{\QQ}{\alpha}{<}{\mu}$ is a witness that $\PP \times \RR$ has a $(\lambda, \kappa, \mu, \DDD)$-filtration.
 Indeed clauses (1)--(3) of Definition \ref{def:filtration} only depend on $\lambda, \kappa, \mu$, and $\DDD$, which satisfy these clauses by hypothesis.
 Also, by hypothesis $\PP$ is ${\lambda}^{+}$-c.c.\@ and ${\forces}_{\PP}{\text{``} \: \check{\RR} \ \text{is} \ {\check{\lambda}}^{+}\text{-c.c.\@''}}$.
 It is a standard fact (e.g.\@ Lemma 5.7 of \cite{Kunen}) that this implies that $\PP \times \RR$ is ${\lambda}^{+}$-c.c.\@
 Also if $\pr{p}{r} \in \PP \times \RR$, then there is $\alpha < \mu$ with $p \in {\PP}_{\alpha}$, whence $\pr{p}{r} \in {\PP}_{\alpha} \times \RR = {\QQ}_{\alpha}$.
 This verifies (4) of Definition \ref{def:filtration}.
 Next for each $\alpha < \mu$, since ${\PP}_{\alpha} \: \cs \: \PP$, ${\QQ}_{\alpha} = {\PP}_{\alpha} \times \RR \: \cs \: \PP \times \RR$.
 Also for each $\xi < \alpha$, we have ${\QQ}_{\xi} = {\PP}_{\xi} \times \RR \subseteq {\PP}_{\alpha} \times \RR = {\QQ}_{\alpha}$ because ${\PP}_{\xi} \subseteq {\PP}_{\alpha}$.
 Finally, for each $\alpha < \mu$, $\lc {\QQ}_{\alpha} \rc = \lc {\PP}_{\alpha} \times \RR \rc < \mu$.
 This concludes the verification.
\end{proof}
The next theorem shows that $\PP$ forces $\uu(\kappa)$ to be bounded by $\mu$ whenever $\PP$ has a $(\lambda, \kappa, \mu, \DDD)$-filtration.
The proof is similar to the proof of Theorem 6 from \cite{sh:304}, though our theorem below is formulated in a slightly more general context.
\begin{Theorem} \label{thm:general}
 Let $\pr{\PP}{{\leq}_{\PP}, {\mathbbm{1}}_{\PP}}$ be a forcing notion.
 Assume that $\lambda, \kappa, \mu$, and $\DDD$ are so that $\pr{\PP}{{\leq}_{\PP}, {\mathbbm{1}}_{\PP}}$ has a $(\lambda, \kappa, \mu, \DDD)$-filtration.
 Assume moreover that $\cf(\kappa) = \kappa$.
 Then $\PP$ forces that every uniform ultrafilter on $\kappa$ that extends $\DDD$ is generated by a set of size at most $\mu$.
 In particular, $\PP$ forces that $\uu(\kappa) \leq \mu$. 
\end{Theorem}
\begin{proof}
 Let $\seq{\PP}{\alpha}{<}{\mu}$ witness that $\pr{\PP}{{\leq}_{\PP}, {\mathbbm{1}}_{\PP}}$ has a $(\lambda, \kappa, \mu, \DDD)$-filtration.
 For each $\alpha < \mu$, let ${A}_{\alpha} = {\PP}^{\lambda}_{\alpha}$ and let $B = {}^{\lambda}{2}$.
 Define ${L}_{\alpha} = {A}_{\alpha} \times B$.
 Then $\lc {L}_{\alpha} \rc < \mu$ because $\mu$ is a strong limit.
 For any $D \in \DDD$, define ${L}_{\alpha, D} = \{\alpha\} \times \{D\} \times {L}^{D}_{\alpha}$.
 Again $\lc {L}_{\alpha, D} \rc < \mu$, and so if $L = \bigcup\{{L}_{\alpha, D}: \pr{\alpha}{D} \in \mu \times \DDD\}$, then $\lc L \rc \leq \mu$.
 Fix a strictly increasing cofinal sequence $\seq{\alpha}{i}{<}{\cf(\mu)}$ in $\mu$.
 
 Suppose $\mathring{A}$ is any $\PP$-name such that ${\forces}_{\PP}{\mathring{A} \subseteq \kappa}$.
 We will associate a member of $L$ to $\mathring{A}$ as follows.
 Since $\PP$ is ${\lambda}^{+}$-c.c.\@, we can find for each $\delta < \kappa$, sequences $\left\langle {p}_{\mathring{A}, \delta, \varepsilon}: \varepsilon < \lambda \right\rangle$ and ${\tau}_{\mathring{A}, \delta}$ such that:
 \begin{enumerate}
  \item
  $\left\langle {p}_{\mathring{A}, \delta, \varepsilon}: \varepsilon < \lambda \right\rangle \in {\PP}^{\lambda}$ and $\left\{{p}_{\mathring{A}, \delta, \varepsilon}: \varepsilon < \lambda \right\}$ is a predense set in $\PP$;
  \item
  ${\tau}_{\mathring{A}, \delta} \in {}^{\lambda}{2}$ and for each $\varepsilon < \lambda$, $\left( {p}_{\mathring{A}, \delta, \varepsilon} \: \forces \: {\delta \in \mathring{A}} \right) \iff \left( {\tau}_{\mathring{A}, \delta}(\varepsilon) = 1 \right)$ and $\left( {p}_{\mathring{A}, \delta, \varepsilon} \: \forces \: {\delta \notin \mathring{A}} \right) \iff \left( {\tau}_{\mathring{A}, \delta}(\varepsilon) = 0 \right)$.
 \end{enumerate}
 For each $\delta < \kappa$ and $\varepsilon < \lambda$, there is $i(\mathring{A}, \delta, \varepsilon) < \cf(\mu)$ with ${p}_{\mathring{A}, \delta, \varepsilon} \in {\PP}_{{\alpha}_{i(\mathring{A}, \delta, \varepsilon)}}$.
 Since $\lambda < \cf(\mu)$, for each $\delta < \kappa$, there exists $i(\mathring{A}, \delta) < \cf(\mu)$ which satisfies $\forall \varepsilon < \lambda\[i(\mathring{A}, \delta, \varepsilon) < i(\mathring{A}, \delta)\]$.
 Now for each $i < \cf(\mu)$, put ${Y}_{i} = \{\delta < \kappa: i = i(\mathring{A}, \delta)\}$.
 Then $\seq{Y}{i}{<}{\cf(\mu)}$ is a partition of $\kappa$.
 Since $\DDD$ is $\cf(\mu)$-indecomposable, there is a set $T \subseteq \cf(\mu)$ such that $\lc T \rc < \cf(\mu)$ and ${\bigcup}_{i \in T}{{Y}_{i}} \in \DDD$.
 Define $D(\mathring{A}) = {\bigcup}_{i \in T}{{Y}_{i}}$.
 $\cf(\mu)$ being a regular cardinal, there is $i(\mathring{A}) < \cf(\mu)$ with $T \subseteq i(\mathring{A})$.
 Define $\alpha(\mathring{A}) = {\alpha}_{i(\mathring{A})} \in \mu$.
 Note that if $\delta \in D(\mathring{A})$, then $i(\mathring{A}, \delta) < i(\mathring{A})$, and so for any $\varepsilon < \lambda$, $i(\mathring{A}, \delta, \varepsilon) < i(\mathring{A}, \delta) < i(\mathring{A})$, whence ${p}_{\mathring{A}, \delta, \varepsilon} \in {\PP}_{{\alpha}_{i(\mathring{A}, \delta, \varepsilon)}} \subseteq {\PP}_{{\alpha}_{i(\mathring{A}, \delta)}} \subseteq {\PP}_{{\alpha}_{i(\mathring{A})}} = {\PP}_{\alpha(\mathring{A})}$.
 Thus we conclude that for each $\delta \in D(\mathring{A})$, $ \left\langle \langle {p}_{\mathring{A}, \delta, \varepsilon}: \varepsilon < \lambda \rangle, {\tau}_{\mathring{A}, \delta} \right\rangle \in {\PP}^{\lambda}_{\alpha(\mathring{A})} \times {}^{\lambda}{2} = {A}_{\alpha(\mathring{A})} \times B = {L}_{\alpha(\mathring{A})}$.
 Therefore $l(\mathring{A}) \in \left\{ \alpha(\mathring{A}) \right\} \times \left\{ D(\mathring{A}) \right\} \times {L}_{\alpha(\mathring{A})}^{D(\mathring{A})} = {L}_{\alpha(\mathring{A}), D(\mathring{A})} \subseteq L$, where
 \begin{align*}
 l(\mathring{A}) = \left\langle \alpha(\mathring{A}), D(\mathring{A}), \left\langle \left\langle \langle {p}_{\mathring{A}, \delta, \varepsilon}: \varepsilon < \lambda \rangle, {\tau}_{\mathring{A}, \delta} \right\rangle: \delta \in D(\mathring{A}) \right\rangle \right\rangle.
 \end{align*}
 \begin{Claim} \label{claim:general}
 Suppose $\mathring{A}$ and $\mathring{B}$ are $\PP$-names such that ${\forces}_{\PP}{\mathring{A} \subseteq \kappa}$ and ${\forces}_{\PP}{\mathring{B} \subseteq \kappa}$.
 Suppose $l = \left\langle \alpha, D, \left\langle \left\langle \langle {p}_{\delta, \varepsilon}: \varepsilon < \lambda \rangle, {\tau}_{\delta} \right\rangle: \delta \in D \right\rangle \right\rangle$ is a member of $L$ such that $l = l(\mathring{A}) = l(\mathring{B})$.
 Then ${\forces}_{\PP}{\mathring{A} \cap D = \mathring{B} \cap D}$.
 \end{Claim}
 \begin{proof}
  Suppose not.
  Then, without loss of generality, there exists $p \in \PP$ and $\delta \in D$ such that $p \: {\forces}_{\PP} \: \delta \in \mathring{A} \setminus \mathring{B}$.
  Since $\delta \in D = D(\mathring{A}) = D(\mathring{B})$, ${\tau}_{\mathring{A}, \delta} = {\tau}_{\delta} = {\tau}_{\mathring{B}, \delta}$ and $\langle {p}_{\mathring{A}, \delta, \varepsilon}: \varepsilon < \lambda \rangle = \langle {p}_{\delta, \varepsilon}: \varepsilon < \lambda \rangle = \langle {p}_{\mathring{B}, \delta, \varepsilon}: \varepsilon < \lambda \rangle$.
  In particular, $\left\{ {p}_{\delta, \varepsilon}: \varepsilon < \lambda \right\}$ is a predense set.
  Choose $\varepsilon < \lambda$ and $q \in \PP$ with $q \: {\leq}_{\PP} \: p, {p}_{\delta, \varepsilon}$.
  It follows from (2) that ${p}_{\delta, \varepsilon} \: {\forces}_{\PP} \: {\delta \in \mathring{A}}$, and hence ${\tau}_{\mathring{A}, \delta}(\varepsilon) = 1$.
  Similarly, ${p}_{\delta, \varepsilon} \: {\forces}_{\PP} \: {\delta \notin \mathring{B}}$, and hence ${\tau}_{\mathring{B}, \delta}(\varepsilon) = 0$.
  However this contradicts ${\tau}_{\mathring{A}, \delta}(\varepsilon) = {\tau}_{\delta}(\varepsilon) = {\tau}_{\mathring{B}, \delta}(\varepsilon)$.  
 \end{proof}  
 Suppose $G$ is a $(\V, \PP)$-generic filter.
 Since $\PP$ is ${\lambda}^{+}$-c.c.\@ and $\cf(\kappa) = \kappa \geq {\lambda}^{+}$, $\kappa$ remains a regular cardinal in $\VG$.
 Suppose $\UUU$ is any uniform ultrafilter on $\kappa$ extending $\DDD$.
 Define
 \begin{align*}
  K = \left\{l \in L: \exists \mathring{A} \in \VP\[{\left( {\forces}_{\PP}{\mathring{A} \subseteq \kappa} \right)}^{\V} \ \text{and} \ l(\mathring{A}) = l \ \text{and} \ \mathring{A}\[G\] \in \UUU \]\right\}.
 \end{align*}
 For each $l \in K$ choose ${\mathring{A}}_{l} \in \VP$ such that ${\left( {\forces}_{\PP}{{\mathring{A}}_{l} \subseteq \kappa} \right)}^{\V}$, $l({\mathring{A}}_{l}) = l$, and ${\mathring{A}}_{l}\[G\] \in \UUU$.
 Define $\XXX = \left\{D \cap {\mathring{A}}_{l}\[G\]: \pr{D}{l} \in \DDD \times K \right\}$.
 Note that if $\pr{D}{l} \in \DDD \times K$, then $D \in \UUU$ because $\UUU$ extends $\DDD$, and ${\mathring{A}}_{l}\[G\] \in \UUU$ by choice of ${\mathring{A}}_{l}$, whence $D \cap {\mathring{A}}_{l}\[G\] \in \UUU$.
 Thus $\XXX \subseteq \UUU$.
 Furthermore $\lc \XXX \rc \leq \mu$ (note that $\mu$ remains a cardinal in $\VG$).
 \begin{Claim} \label{claim:general1}
  $\XXX$ generates $\UUU$.
 \end{Claim}
 \begin{proof}
  Since $\XXX \subseteq \UUU$, $\{B \subseteq \kappa: \exists A \in \XXX\[A \subseteq B\]\} \subseteq \UUU$.
  Suppose that $C \in \UUU$.
  Find $\mathring{C} \in \VP$ with $C = \mathring{C}\[G\]$.
  Also find $p \in G$ so that ${\left( p \: {\forces}_{\PP} \: {\mathring{C} \subseteq \kappa} \right)}^{\V}$.
  In $\V$, applying the maximal principle, we can find a $\PP$-name $\mathring{B}$ so that ${\forces}_{\PP}{\mathring{B} \subseteq \kappa}$ and $p \: {\forces}_{\PP} \: {\mathring{B} = \mathring{C}}$.
  Let $l = l(\mathring{B}) \in L$.
  In $\VG$, $\mathring{B}\[G\] = \mathring{C}\[G\] = C \in \UUU$, and $\mathring{B}$ is a witness that $l \in K$.
  Hence ${\mathring{A}}_{l}$ is defined and $l({\mathring{A}}_{l}) = l = l(\mathring{B})$. 
  It follows from Claim \ref{claim:general} that for some $D \in \DDD$, ${\mathring{A}}_{l}\[G\] \cap D = \mathring{B}\[G\] \cap D = C \cap D$.
  Since $\pr{D}{l} \in \DDD \times K$, $D \cap {\mathring{A}}_{l}\[G\] \in \XXX$.
  So $C \cap D \in \XXX$, and since $C \cap D \subseteq C$, $C \in \{B \subseteq \kappa: \exists A \in \XXX\[A \subseteq B\]\}$.
  Thus $\UUU = \{B \subseteq \kappa: \exists A \in \XXX\[A \subseteq B\]\}$, as needed.
 \end{proof}
 This proves that $\UUU$ is generated by a set of size at most $\mu$.
 Since $\DDD$ is a uniform filter on $\kappa$, there is at least one uniform ultrafilter on $\kappa$ extending $\DDD$.
 Therefore $\uu(\kappa) \leq \mu$ in $\VG$.
\end{proof}
The proof of Theorem \ref{thm:general} shows that the results in Section 4 of \cite{sh:304} can also be obtained from the assumption that $\PP$ has a $(\lambda, \kappa, \mu, \DDD)$-filtration.
Note also that the condition that ${\PP}_{\alpha} \: \cs \: \PP$ for every $\alpha < \mu$ is not used in the proof of Theorem \ref{thm:general}.
Hence this theorem can be proved under a weaker formulation of Definition \ref{def:filtration}.
We leave it to the interested reader to formulate the optimal hypotheses under which the proof of Theorem \ref{thm:general} can be carried out. 
\section{Small ultrafilter number at the continuum} \label{sec:uc}
 Several posets of the form $\Fn(I, J, \chi)$ as well as products of such posets have a $(\lambda, \kappa, \mu, \DDD)$-filtration for suitable values of the cardinals $\lambda, \kappa$, and $\mu$, and any uniform $\cf(\mu)$-indecomposable filter $\DDD$ on $\kappa$.  
 \begin{Lemma} \label{lem:Pfiltration}
  Suppose that $\lambda, \kappa, \mu$, and $\DDD$ satisfy (1)--(3) of Definition \ref{def:filtration}.
  Then $\Fn(\mu\times\lambda, 2, \lambda)$ has a $(\lambda, \kappa, \mu, \DDD)$-filtration.
 \end{Lemma}
 \begin{proof}
  For each $\alpha < \mu$, define ${\PP}_{\alpha}$ to be $\Fn(\alpha\times\lambda, 2, \lambda)$.
  We will check that the sequence $\seq{\PP}{\alpha}{<}{\mu}$ witnesses that there is a $(\lambda, \kappa, \mu, \DDD)$-filtration.
  Clauses (1)--(3) of Definition \ref{def:filtration} are already satisfied by hypothesis.
  It is well-known (see Lemma 6.10 of \cite{Kunen}) that $\Fn(\mu\times\lambda, 2, \lambda)$ is ${\left( {2}^{< \lambda} \right)}^{+}$-c.c.
  This means that $\Fn(\mu\times\lambda, 2, \lambda)$ is ${\left( \lambda \right)}^{+}$-c.c.\@ because ${\lambda}^{< \lambda} = \lambda$.
  Also for each $p \in \Fn(\mu\times\lambda, 2, \lambda)$, there exists $\alpha < \mu$ with $p \in \Fn(\alpha\times\lambda, 2, \lambda) = {\PP}_{\alpha}$ because $\lambda < \cf(\mu)$.
  Next for any $\alpha < \mu$, $\alpha\times\lambda \subseteq \mu\times\lambda$, and so ${\PP}_{\alpha} = \Fn(\alpha\times\lambda, 2, \lambda) \: \cs \: \Fn(\mu\times\lambda, 2, \lambda)$.
  Similarly if $\xi < \alpha$, then ${\PP}_{\xi} = \Fn(\xi\times\lambda, 2, \lambda) \: \cs \: \Fn(\alpha\times\lambda, 2, \lambda) = {\PP}_{\alpha}$.
  Finally for each $\alpha < \mu$, $\lc {\PP}_{\alpha} \rc < \mu$ because $\mu$ is a strong limit cardinal.
  Therefore (1)--(6) of Definition \ref{def:filtration} are satisfied.
 \end{proof}
 \begin{Lemma} \label{lem:PRfiltration}
  Suppose that $\lambda, \kappa, \mu$, and $\DDD$ satisfy (1)--(3) of Definition \ref{def:filtration}.
  Suppose moreover that $\theta$ is an infinite cardinal such that $\theta < \lambda$, $\theta$ is regular, ${2}^{< \theta} = \theta$, ${\kappa}^{\theta} = \kappa$, and $\cf(\kappa) = \kappa$.
  Let $\PP = \Fn(\mu\times\lambda, 2, \lambda)$ and $\RR = \Fn(\kappa\times\theta, 2, \theta)$. 
  Then  $\PP \times \RR$ has a $(\lambda, \kappa, \mu, \DDD)$-filtration.
 \end{Lemma}
 \begin{proof}
  We will check the hypotheses of Lemma \ref{lem:productfiltration}.
  Firstly, by Lemma \ref{lem:Pfiltration} $\PP$ has a $(\lambda, \kappa, \mu, \DDD)$-filtration.
  Using the fact that $\mu$ is a strong limit cardinal, it is easy to verify that $\lc \RR 
  \rc < \mu$.
  Finally suppose that $G$ is $(\V, \PP)$-generic.
  Since $\PP$ is $\lambda$-closed in $\V$, it follows that $\RR$ is still $\Fn(\kappa\times\theta, 2, \theta)$ as calculated in $\VG$.
  Similarly in $\VG$, ${2}^{< \theta} = \theta$ holds, and so $\RR$ is ${\theta}^{+}$-c.c.\@ in $\VG$.
  Therefore in $\V$, ${\forces}_{\PP}{\text{``} \: \check{\RR} \ \text{is} \ {\check{\lambda}}^{+}\text{-c.c.\@''}}$.
  Hence $\PP \times \RR$ also has a $(\lambda, \kappa, \mu, \DDD)$-filtration by Lemma \ref{lem:productfiltration}.
 \end{proof}
 We are now able to show that if there is a measurable cardinal $\kappa$, then for any regular cardinal $\theta < \kappa$ satisfying ${2}^{<\theta} = \theta$, it is possible to force $\uu({2}^{\theta}) < {2}^{{2}^{\theta}}$.
 \begin{Theorem} \label{thm:main1}
  Suppose that $\theta$, $\lambda, \kappa$, and $\mu$ are infinite cardinals satisfying $\theta < \lambda < \cf(\mu) < \kappa < \mu$, that $\theta$ is regular, and that ${2}^{< \theta} = \theta$.
  Assume also that $\mu$ is a strong limit cardinal and that ${\lambda}^{< \lambda} = \lambda$.
  Suppose moreover that $\kappa$ is measurable and that $\DDD$ is a normal measure on $\kappa$.
  Then there is a cofinality preserving extension in which ${2}^{\theta} = \kappa$, $\uu(\kappa) \leq \mu$, and ${2}^{\kappa} = {\mu}^{\kappa} > \mu$.
 \end{Theorem}
 \begin{proof}
  Note that since $\kappa$ is measurable and $\DDD$ is a normal measure on $\kappa$, $\cf(\kappa) = \kappa$, ${\kappa}^{\theta} = \kappa$, and $\DDD$ is a uniform $\cf(\mu)$-indecomposable ultrafilter on $\kappa$.
  Let $\PP = \Fn(\mu\times\lambda, 2, \lambda)$ and $\RR = \Fn(\kappa\times\theta, 2, \theta)$.
  It is well-known that $\PP \times \RR$ is cofinality preserving.
  By Lemma \ref{lem:PRfiltration}, $\PP \times \RR$ has a $(\lambda, \kappa, \mu, \DDD)$-filtration.
  Suppose $H$ is a $(\V, \PP \times \RR)$-generic filter.
  By Theorem \ref{thm:general}, $\uu(\kappa) \leq \mu$ holds in $\V\[H\]$.
  By standard arguments (see proof of Theorem 6.18 in \cite{Kunen}), $\V\[H\]$ satisfies ${2}^{\theta} = \kappa$ and ${2}^{\lambda} = \mu$.
  Since cofinalities and cardinals are preserved, we have ${2}^{\kappa} = {\left( {2}^{\lambda} \right)}^{\kappa} = {\mu}^{\kappa} \geq {\mu}^{\cf(\mu)} > \mu$.
 \end{proof}
 \begin{Cor} \label{cor:continuum}
  It is consistent relative to a measurable cardinal that there is a uniform ultrafilter on the reals which is generated by fewer than ${2}^{{2}^{{\aleph}_{0}}}$ many sets.
 \end{Cor}
 \begin{proof}
  Apply Theorem \ref{thm:main1} with $\theta = \omega$.
 \end{proof}
Note that the fact that $\mu$ is a singular cardinal with cofinality less than $\kappa$ plays a crucial role in these proofs.
In Corollary \ref{cor:continuum}, if we assume that $\GCH$ holds in the ground model and pick the minimal values of $\lambda$ and $\mu$ required to run the proof, then in the resulting model, ${2}^{{\aleph}_{0}} = \kappa$ and ${2}^{{2}^{{\aleph}_{0}}} > {\kappa}^{+{\aleph}_{2}}$.
\section{Small ultrafilter number at ${\aleph}_{\omega+1}$} \label{sec:alephomega+1}
The continuum is a weakly inaccessible cardinal in the model constructed in the previous section.
In this section we will get a model where $\uu(\kappa) < {2}^{\kappa}$ for a $\kappa$ which is well below the first weakly inaccessible cardinal, namely ${\aleph}_{\omega+1}$.
However we must start with a supercompact cardinal.

Thomas and Shelah~\cite{sh:304} considered the following statement for regular cardinals $\kappa$:
\begin{align*}
&\text{If} \ G \ \text{is any subgroup of} \  \mathord{\mathrm{Sym}(\kappa)} \ \text{with} \ [\mathord{\mathrm{Sym}(\kappa)}:G] < {2}^{\kappa}\text{, then there exists} \tag{${\ast}_{\kappa}$} \label{eq:ast} \\ 
&\Delta \subseteq \kappa \ \text{such that} \ \lc \Delta \rc < \kappa \ \text{and} \ {S}_{\left( \Delta \right)} \ \text{is a subgroup of} \ G.
\end{align*}
Here $\mathord{\mathrm{Sym}(\kappa)}$ is the symmetric group on $\kappa$ and ${S}_{\left( \Delta \right)}$ denotes the pointwise stabilizer of the set $\Delta \subseteq \kappa$.
It turns out that $\left( {\ast}_{{\aleph}_{0}} \right)$ is a theorem of $\ZFC$.
In \cite{sh:304}, Shelah and Thomas used a supercompact cardinal to produce a model where $\left( {\ast}_{{\aleph}_{\omega+1}} \right)$ fails.
We show below that $\uu({\aleph}_{\omega+1}) < {2}^{{\aleph}_{\omega+1}}$ in this model constructed by Shelah and Thomas, and moreover our proof is quite similar to their argument that $\left( {\ast}_{{\aleph}_{\omega+1}} \right)$ fails.
However we are not aware of any direct connection between $\left( {\ast}_{{\aleph}_{\omega+1}} \right)$ and $\uu({\aleph}_{\omega+1})$.
It would be especially interesting if the failure of $\left( {\ast}_{\kappa} \right)$ implied $\uu(\kappa) < {2}^{\kappa}$ for some uncountable regular $\kappa$.

In order to apply Theorem \ref{thm:general} with $\kappa = {\aleph}_{\omega+1}$, it must be possible to find uniform filters on ${\aleph}_{\omega+1}$ that are ${\aleph}_{n}$-indecomposable for some $n < \omega$.
A by now classical theorem of Ben-David and Magidor~\cite{indecomposable} says that it is consistent relative to a supercompact cardinal to have a uniform ultrafilter on ${\aleph}_{\omega+1}$ which is ${\aleph}_{n}$-indecomposable for all $0 < n < \omega$.
\begin{Theorem}[Ben-David and Magidor~\cite{indecomposable}] \label{thm:bendavidmagidor}
 Assume that there is a supercompact cardinal.
 There is a forcing extension in which $\GCH$ holds and there is a uniform ultrafilter on ${\aleph}_{\omega+1}$ which is ${\aleph}_{n}$-indecomposable for all $0 < n < \omega$.
\end{Theorem}
The proof of the next theorem is now just a matter of combining Theorem \ref{thm:bendavidmagidor} with Theorem \ref{thm:general} and Lemma \ref{lem:Pfiltration}.
\begin{Theorem} \label{thm:alephomega+1}
 Assume that there is a supercompact cardinal.
 Then there is a forcing extension in which $\uu({\aleph}_{\omega+1}) < {2}^{{\aleph}_{\omega+1}}$.
\end{Theorem}
\begin{proof}
 By Theorem \ref{sec:alephomega+1}, we can pass to a forcing extension $\V'$ in which $\GCH$ holds and there exists a uniform ultrafilter $\DDD$ on ${\aleph}_{\omega+1}$ which is ${\aleph}_{n}$-indecomposable for all $0 < n < \omega$.
 Working in $\V'$, put $\kappa = {\aleph}_{\omega+1}$ and choose $\lambda$ and $\mu$ so that (1)--(3) of Definition \ref{def:filtration} are satisfied.
 In fact, since $\GCH$ holds in $\V'$, we can simply choose $\lambda = {\aleph}_{0}$ and $\mu = {\aleph}_{{\omega}_{1}}$.
 By Lemma \ref{lem:Pfiltration}, $\PP = \Fn(\mu\times\lambda, 2, \lambda)$ has a $(\lambda, \kappa, \mu, \DDD)$-filtration.
 Let $G$ be $(\V', \PP)$-generic.
 By standard arguments, ${2}^{\lambda} = \mu$ in $\V'\[G\]$.
 By Theorem \ref{thm:general} and by the fact the all cofinalities and cardinals are preserved between $\V'$ and $\V'\[G\]$, ${\aleph}_{\omega+1} = \kappa$, $\uu(\kappa) \leq \mu$, and ${2}^{\kappa} = {\left( {2}^{\lambda} \right)}^{\kappa} = {\mu}^{\kappa} \geq {\mu}^{\cf(\mu)} > \mu$ in $\V'\[G\]$.
\end{proof}
The reader will again notice the crucial role played by the fact that $\mu$ is a singular cardinal whose cofinality is smaller than $\kappa$.
Choosing the minimal values for $\lambda$ and $\mu$ that are allowed by the proof, as we have done above, still results in a model where ${2}^{{\aleph}_{\omega+1}} > {\aleph}_{{\omega}_{1}}$.
This is very unlikely to be sharp.
It ought to be possible to produce models where $\uu({\aleph}_{\omega+1}) = {\aleph}_{\omega+2} < {\aleph}_{\omega+3} = {2}^{{\aleph}_{\omega+1}}$.

It is not difficult to combine the proof of Theorem \ref{thm:alephomega+1} with the proof of Theorem \ref{thm:main1} to produce a model where ${2}^{{\aleph}_{0}} = {\aleph}_{\omega+1}$ and $\uu({\aleph}_{\omega+1}) < {2}^{{\aleph}_{\omega+1}}$.
One would then need to choose $\mu$ to be ${\aleph}_{{\omega}_{2}}$ (or bigger).
Details are left to the reader.
\section{Remarks and Questions} \label{sec:q}
As mentioned in Section \ref{sec:intro}, the models constructed in this paper have several features that are likely to be accidental rather than essential.
The first such feature is the use of large cardinals.
\begin{Question} \label{q:1}
 What is the consistency strength of the inequality $\uu({2}^{{\aleph}_{0}}) < {2}^{{2}^{{\aleph}_{0}}}$ or of $\uu({\aleph}_{\omega+1}) < {2}^{{\aleph}_{\omega+1}}$?
\end{Question}
We are not aware that these statements have any large cardinal strength.
The next question is about how large ${2}^{\kappa}$ needs to be for $\uu(\kappa) < {2}^{\kappa}$ to be consistent.
We pose this question in a very weak form below.
\begin{Question} \label{q:2}
 Is the following statement consistent relative to large cardinals: There exists an uncountable regular cardinal $\kappa$ such that $\kappa$ is smaller than the first weakly inaccessible cardinal and $\uu(\kappa) = {\kappa}^{+} < {\kappa}^{++} = {2}^{\kappa}$?
\end{Question}
Finally of course the method in this paper is not applicable to any of the ${\aleph}_{n}$.
\begin{Question} \label{q:3}
 Is it consistent to have $\uu({\aleph}_{n}) < {2}^{{\aleph}_{n}}$, for some $0 < n < \omega$?
\end{Question}
\input{smallu.bbl}

\end{document}

%% file: smallu.bbl
\def\polhk#1{\setbox0=\hbox{#1}{\ooalign{\hidewidth
  \lower1.5ex\hbox{`}\hidewidth\crcr\unhbox0}}}
\providecommand{\bysame}{\leavevmode\hbox to3em{\hrulefill}\thinspace}
\providecommand{\MR}{\relax\ifhmode\unskip\space\fi MR }
\providecommand{\MRhref}[2]{%
  \href{http://www.ams.org/mathscinet-getitem?mr=#1}{#2}
}
\providecommand{\href}[2]{#2}